\documentclass{amsart}

\usepackage{enumerate}
\usepackage{amssymb}
\usepackage{amsfonts}
\usepackage{latexsym}
\usepackage{amsmath}
\usepackage{euscript}
\usepackage{graphics}
\usepackage{graphicx}
\usepackage{tikz}
\usetikzlibrary{spy}
\graphicspath{ {./images/} }

\newtheorem{theorem}{Theorem}[section]
\newtheorem{proposition}[theorem]{Proposition}
\newtheorem{corollary}[theorem]{Corollary}
\newtheorem{lemma}[theorem]{Lemma}
\newtheorem{definition}[theorem]{Definition}
\theoremstyle{remark}
\newtheorem{example}[theorem]{Example}

\newtheorem{remark}[theorem]{Remark}

\numberwithin{equation}{section}

\usepackage[colorinlistoftodos]{todonotes}

\begin{document}
\title[An Exploration of Lacunary Orthogonal Polynomials]{An Exploration of Lacunary Orthogonal Polynomials Associated with the Geronimus Transformation}


\author[R.~Bailey]{R.~Bailey}
\address{
RB,
Department of Mathematical Sciences\\
Bentley University\\
175 Forest Street\\
 Waltham, MA 02452, USA}
\email{rbailey@bentley.edu}

\author[R.~Gavrilov]{R.~Gavrilov}
\address{
RG,
Department of Statistics \\
Columbia University\\
116th and Broadway\\
New York, NY 10027, USA}
\email{rg3801@columbia.edu}


\begin{abstract}
{In this note, we present two perturbations of sequences of orthogonal polynomials related to two iterations of the Geronimus transformation, with the feature that the resulting sequences of polynomials are missing a single degree. We relate these sequences to quasi-orthogonal polynomials of order 2 and investigate properties including recurrence relations, completeness, and associated differential equations.}
\end{abstract}

\maketitle

\section{introduction}
Consider an orthogonal polynomial sequence (OPS) $\{P_n(x)\}_{n=0}^{\infty}$  orthogonal with respect to a  measure $\mu(x)$ supported on an infinite subset of $\mathbb{R}$, and let $m$ be a fixed integer. It is known that defining a new sequence of polynomials $\{Q_n(x)\}_{n=0}^{\infty}$ by
\[
Q_n(x):=a_{n,n}P_n(x)+a_{n,n-1}P_{n-1}(x)+\dots+a_{n,m}P_{n-m}(x), \quad n \geq 0,
\] where $a_{n,j}$ are real numbers such that $a_{n,m}\neq 0$ is a \textit{quasi-orthogonal}  polynomial sequence of order $m$, which is a sequence of polynomials satisfying 
\[
\int_{-\infty}^{\infty}p(x)Q_n(x)d\mu(x)=0 
\] where $p(x)$ is a polynomial of degree at most $n-m-1$. Note that we define $P_{n-m}(x):=0$ for $n-m<0$ (one may refer to \cite{ChiharaQuasi} for more information on quasi-orthogonal polynomials). Our focus will be polynomials of the form
\begin{equation}\label{eq:quasi}
Q_n(x)=P_n(x)+B_{n-1}(\kappa)P_{n-1}(x)+A_{n-2}(\kappa)P_{n-2}(x), \,A_n(\kappa)\neq 0, \quad n\geq 2,
\end{equation} which are associated with two iterations of a  transformation called the \textit{Geronimus transformation} at a complex number $\kappa$. The Geronimus transformation is a fundamental tool in the theory of OPSs, corresponding to the transformation of measures $\mu(x)\rightarrow \frac{\mu(x)}{x-\kappa}+N\delta_{\kappa}$ where $\delta_{\kappa}$ is the Delta function supported at $\kappa$  and $N$ is a real number.

Dating back to foundational work by Geronimus in the early 1940s, researchers have long examined how such modifications of measures affect the corresponding polynomial sequences. In \cite{maroni}, these transformation are put into the framework of linear functionals and in \cite{Zhedanov97}, a connection to rational transformations of the Stieltjes function is discussed. More recently, the work in \cite{DGM14} and  \cite{DM13} demonstrated that iterated Geronimus transformations are deeply tied to Sobolev-type orthogonal polynomials. Furthermore, conditions on when these transformations yield an OPS at a complex number $\kappa$ were recently investigated in \cite{BD23}.

When viewed through the lens of Jacobi matrices, one can see the deep connection between the Geronimus transformation and the discrete analog of the Darboux transformation  (see \cite{MarcellanBueno} for example). Beyond this discrete framework, orthogonal polynomials also play a crucial role in continuous Darboux transformations via differential operators (see \cite{Adler} and \cite{DDM} for more information on the continuous Darboux transformation). This interplay has gained widespread attention in the orthogonal polynomial community, driving the development of Exceptional Orthogonal Polynomials (XOPs) which are complete families of orthogonal polynomials that are missing finitely many degrees in the sequence and are eigenfunctions of second order differential operators. For example, it was shown in \cite{DEK92} that applying two iterations of the continuous Darboux transformation results in a differential operator with eigenfunctions related to the exceptional Hermite polynomials. Additionally, it was demonstrated by Garc\'ia-Ferrero, G\'omez-Ullate and Milson \cite{GGM19} that applying sequences of continuous Darboux transformations to classical orthogonal polynomials leads to XOPs. More precisely, a rational Darboux transformation applied to a classical OPS produces an XOP sequence orthogonal with respect to $\frac{W_0(x)}{(\eta(x))^2}$, where $W_0(x)$ is the original orthogonality weight and $\eta(x)$ is a polynomial. Since classical OPS are solutions to both differential and difference equations, Darboux transformation is also fundamental in the study of bispectral problems as is discussed in \cite{CG24}, \cite{CGZ24} and \cite{GH97}, for example. For further reading on XOPs, we refer the reader to \cite{D21} and \cite{GM20}.

Motivated by the denominator function related to two iterations of the Geronimus transformation at $\kappa=i\sqrt{c}$ and $\overline{\kappa}$, this paper aims to explore how one may modify conditions which are needed to allow for an OPS corresponding to $\frac{d\mu(x)}{c+x^2}$, and what happens to the polynomials under these modifications. We note that the denominator of the modified measure we study is not the square of a polynomial, thus differs from the denominators of the weight functions for XOPs.

The structure of the paper is the following. Section \ref{sec:setup} sets the notation and provides the tools we will use for proofs later in the paper.  Section \ref{sec:order 1} explores the first perturbation which is an alteration of the coefficient of $P_1(x)$ in \eqref{eq:quasi}. We show that orthogonality may still be preserved, at the cost of losing a polynomial in the sequence. An explicit example utilizing the monic Chebyshev polynomials of the first kind is also given. In Section \ref{sec:quasi}, we discuss how these lacunary orthogonal polynomial sequences (sequences of orthogonal polynomials which are missing finitely many degrees) are related to quasi-orthogonal polynomials of order 2. And finally, in Section \ref{section:limit}, we investigate the second perturbation which is the limiting case of the measure $\frac{d\mu(x)}{c+x^2}$ when $c$ tends to zero and results in the loss of a degree zero polynomial. Focusing on the Hermite, Chebyshev and Legendre polynomials, we find associated second-order differential equations and discuss completeness.

\section{the set up}\label{sec:setup} The following properties of orthogonal polynomials will be used in the remainder of the paper so we include them here for reference. Below we use $\mathbb{N}$ to denote the set of nonnegative integers.
\begin{definition}[Orthogonal Polynomial Sequence]\label{def:OPS}
Let $\mu(x)$ be a measure supported on an infinite subset of $\mathbb{R}$ with finite moments. A sequence of polynomials $\{P_n(x)\}_{n=0}^{\infty}$ is an Orthogonal Polynomial Sequence (OPS) with respect to $\mu(x)$ if
\begin{enumerate}
\item $P_n(x)$ is a polynomial of degree $n$ for $n=0,1,2,\dots$
    \item $\displaystyle \int_{-\infty}^{\infty}P_n(x)P_m(x)d\mu(x)=0 \text{ for } n\neq m$,
    \item $\displaystyle \int_{-\infty}^{\infty}P^2_n(x)d\mu(x)\neq 0  $ for any $n$.
\end{enumerate}
\end{definition}

It is well known that monic OPSs satisfy the following three-term recurrence relation.
\begin{proposition}[Theorem 4.1 \cite{Chihara}]\label{prop:recurrence}
Let $\{P_{n}(x)\}_{n=0}^{\infty}$ be monic OPS with respect to $\mu(x)$. Then there exist real constants $c_{n}$ and $\lambda_{n}> 0$ such that:
\begin{equation}\label{eq:recrel}
P_{n}(x)=(x-c_{n})P_{n-1}(x)-\lambda_{n}P_{n-2}(x), \quad n=1,2,3,... 
\end{equation}
with the conventional that $P_{-1}(x)=0$. 
\end{proposition}
Given an OPS, we can recover the coefficients of the three-term recurrence relation in the following way.
\begin{proposition}[Theorem 4.2 \cite{Chihara}]\label{prop:reccoeff}
    With reference to Proposition \ref{prop:recurrence}, the following are valid for $n\ge 1$:
\begin{enumerate}
    \item $\lambda_{n+1}=\displaystyle \frac{\int_{-\infty}^{\infty} P_{n}^{2}(x)d\mu(x)}{\int_{-\infty}^{\infty} P_{n-1}^{2}(x)d\mu(x)}$;\quad $\displaystyle c_{n}=\frac{\int_{-\infty}^{\infty} xP_{n-1}^{2}(x)d\mu(x)}{\int_{-\infty}^{\infty} P_{n-1}^{2}(x)d\mu(x)}$
   \item $\displaystyle \int_{-\infty}^{\infty} P_{n}^{2}(x)d\mu(x)=\lambda_{1}\lambda_{2}\cdot\cdot\cdot\lambda_{n+1}$, \text{ where } $\lambda_{1}=\int_{-\infty}^{\infty}d\mu(x)$.
  
\end{enumerate}
\end{proposition}
In what follows, we will consider symmetric measures on $\mathbb{R}$, thus, the next proposition provides useful insight into the corresponding OPS in this case.
\begin{proposition}[Theorem 4.3 \cite{Chihara}]\label{prop:symmetric}
    Let $\{P_{n}(x)\}_{n=0}^{\infty}$ be the monic OPS with respect to $\mu(x)$. Then the following are equivalent:
\begin{enumerate}
    \item $\mu(x)$ is symmetric on symmetric subsets of $\mathbb{R}$;
    \item $P_{n}(-x)=(-1)^{n}P_{n}(x)$ for $n\ge 0$;
    \item In the corresponding three-term recurrence relation, $c_{n}=0$ ($n\ge 1$).
\end{enumerate}
\end{proposition}

Additionally, if $\mu(x)$ is symmetric, we know what form the third degree polynomial in the sequence takes, which will be useful for our analysis later.
\begin{lemma}\label{prop:p3}Let $\{P_n(x)\}_{n=0}^{\infty}$ be a sequence of monic orthogonal polynomials with respect to a symmetric measure $\mu(x)$ supported on a symmetric subset of $\mathbb{R}$. Then $P_3(x)=x^3+ax$ where, in particular, $a\neq 0$.
\end{lemma}
\begin{proof} Recall that if $\{P_n(x)\}_{n=0}^{\infty}$ is a sequence of monic orthogonal polynomials, then Then there exist constants $c_n$ and $\lambda_n>0$ such that
\begin{equation}\label{def:rec2}
xP_{n-1}(x)=P_{n}(x)+c_nP_{n-1}(x)+\lambda_nP_{n-2}(x)
\end{equation} where $P_{-1}:=0$. If $\mathcal{L}$ is positive-definite then for $n \geq 1$, we have that $c_n\in \mathbb{R}$ and $\lambda_{n+1}>0$. 
The coefficients $c_n$ are as in given Proposition \ref{prop:reccoeff}.

Since $\mu(x)$ is symmetric, the numerator of $c_n$ is zero and hence $c_n=0$ for all $n=1,2,\dots$. Thus, $P_1(x)=(x-c_1)P_0(x)=x-c_1=x$. Similarly,
\begin{align*}
P_2(x)&=(x-c_2)P_1(x)-\lambda_2P_0(x)\\
&=xP_1(x)-\lambda_2\\
&=x^2-\lambda_2
\end{align*} so that
\begin{equation}\label{eq:p3}
P_3(x)=(x-c_3)P_2(x)-\lambda_3P_1(x)=xP_2(x)-\lambda_3x\\
=x^3-(\lambda_2+\lambda_3)x.
\end{equation}Since $\lambda_n >0$ for all $n=1,2,\dots$, we have that $a=-(\lambda_2+\lambda_3)\neq 0.$

We emphasis that that the important part of this lemma is the fact that $a \neq 0$.

\end{proof}

\section{associated lacunary orthogonal polynomials of codimension 1}\label{sec:order 1}
As mentioned previously, we will explore two perturbations of orthogonal polynomials corresponding to two iterations of the Geronimus transformation. Our focus will be the the particular case the transformed measure is $\frac{d\mu(x)}{c+x^2}$ for a positive, real number $c$, and will restrict our analysis to when $\mu(x)$ is a symmetric measure which is supported on a  symmetric, infinite subset of $\mathbb{R}$.

In the first perturbation we study, we focus on the following setting. Suppose we have a subspace of $L^2\left(d\mu, \mathbb{R}\right)$ that explicitly excludes degree one polynomials. This constraint arises naturally when simulating a quantum system possessing a forbidden transition state. Similar constraints appear in physical systems such as those that dictate zero flux across a boundary, require the net momentum to be zero, or exhibit symmetry forcing a critical point (such as a maximum or stagnation point) at the center of the model. These situations impose the condition $f'(0) = 0$, which inherently bans linear terms.

Our goal is to enforce this omission of linear polynomials and construct a sequence of orthogonal polynomials under this constraint. While the omission of a single degree has been studied via XOPs, standard \(X_{1}\)-sequences omit the degree 0 polynomial rather than the degree 1 polynomial. Furthermore, because the weight function of an XOP sequence features a denominator defined by the square of a polynomial, our specific setting falls outside the conventional XOP framework.

Following the notation of bilinear forms in \cite{DGM14}, let $[\cdot, \cdot]_{\phi}$ be the symmetric bilinear form defined by 
\[
[\phi(x)f(x),g(x)]_{\phi}=[f(x),\phi(x)g(x)]_{\phi}=\int_{-\infty}^{\infty}f(x)g(x)d\mu(x)
\] where $\phi(x)=c+x^2$.
The explicit representation for $[\cdot, \cdot]_{\phi}$ is given in Proposition 1 of \cite{DGM14}. In the same paper, it was shown that the sequence $\{R_n(x)\}_{n=0}^{\infty}$ is an OPS with respect to $[\cdot, \cdot]_{\phi}$  if and only if for $n\geq 2$, they take the form
\begin{equation}\label{eq:gernomimus}
R_n(x)=\frac{1}{d^*_n}\begin{vmatrix}
P_n(x) & [P_n(x), 1]_{\phi}&[P_n(x),x]_{\phi}\\
P_{n-1}(x)&[P_{n-1}(x),1]_{\phi}&[P_{n-1}(x),x]_{\phi}\\
P_{n-2}(x)&[P_{n-2}(x),1]_{\phi}&[P_{n-2}(x),x]_{\phi}
\end{vmatrix},
\end{equation} where \[d^*_n=\begin{vmatrix}
 [P_{n-1}(x),1]_{\phi}&[P_{n-1}(x),x]_{\phi}\\
[P_{n-2}(x),1]_{\phi}&[P_{n-2}(x),x]_{\phi}
\end{vmatrix}.\]
Since $\mu(x)$ is symmetric, \eqref{eq:gernomimus} reduces to 
\[
R_n(x)=P_n(x)-\frac{[P_n(x),x]_{\phi}}{[P_{n-2}(x),x]_{\phi}}P_{n-2}(x)=P_n(x)+A_{n-2}P_{n-2}(x).
\]

To enforce the condition that our sequence omit linear polynomials, we will alter the coefficient of $P_{1}(x)$ in the definition of $R_3(x)$. Indeed, if we pick $A_1\neq -\frac{[P_3,x]_{\phi}}{[P_{1},x]_{\phi}}$ then $R_3(x)$ will not be orthogonal to $x$. 
However, if there were a degree 1 polynomial in the sequence $\{R_n(x)\}$, it would have to be $r(x)=x$ by Proposition \ref{prop:symmetric}, thus the sequence $\{R_n(x)\}$ cannot contain a degree 1 polynomial. 
We will show that it is still possible to make a choice of $A_1$ such that the remaining polynomials are orthogonal. Since there is one polynomial missing from the sequence, we say that the sequence  has \textit{codimension 1}.

 Let us now formally define polynomials $R_n(x;c)$ in the following way:

\begin{equation}\label{eq:Rn_2term}
    R_n(x;c):=\begin{cases}
    1, &n=0\\
     P_{n}(x)+A_{n-2}P_{n-2}(x),& n\geq 2,
        \end{cases}
   \end{equation} such that
for $n\geq 2$,
\begin{equation}\label{cond:1}
   \int_{-\infty}^{\infty} R_0(x;c)R_n(x;c)\frac{d\mu(x)}{c+x^2}=0
    \end{equation}
for $n\geq 4$,
\begin{equation}\label{cond:2}
    \int_{-\infty}^{\infty}  R_3(x;c)R_n(x;c)\frac{d\mu(x)}{c+x^2}=0,
    \end{equation}
and $A_1:=c$.

\begin{remark}
    We remark that the coefficients $A_n=A_n(c)$ depend on $c$, but we omit this in our notation for simplicity.
\end{remark}
Note that by definition, the coefficients $A_n$ (given that they exist) will always be real numbers, thus $R_n(x)$ is a real polynomial of degree $n$.

The polynomials defined by $R_n(x)=P_{n+2}(x)+B_nP_n(x)+A_{n-2}P_{n-2}(x)$ related to 4 iterations of the Geronimus transformation were studied in \cite{Bailey2024}, where it was shown that specific conditions must be satisfied so that the coefficients $A_n$ and $B_n$ exist. In the case we now consider, if $A_1=c$, then the $A_n$'s exist regardless of the choice of symmetric OPS.

\begin{theorem}\label{thrm:an_1} Let $\{P_n(x)\}_{n=0}^{\infty}$ be an OPS with respect to $\mu(x)$ and define polynomials $\{R_n(x;c)\}_{n\in \mathbb{N}\setminus\{1\}}$ by \eqref{eq:Rn_2term}. If $A_1=c$, then $A_n$ exists for all $n=0,1,2,\dots$
\end{theorem}

\begin{proof} 
Note that $A_n$ takes the form
\begin{equation}\label{eq:an_1}
\displaystyle A_n=\begin{cases}\frac{-\int_{-\infty}^{\infty}P_{n+2}(x)\frac{
d\mu(x)}{c+x^2}}{ \int_{-\infty}^{\infty}P_{n}(x)\frac{
d\mu(x)}{c+x^2}}& \text{ $n$ even,}\\
\frac{-\int_{-\infty}^{\infty}R_3(x)P_{n+2}(x)\frac{
d\mu(x)}{c+x^2}}{ \int_{-\infty}^{\infty}R_3(x)P_{n}(x)\frac{
d\mu(x)}{c+x^2}}&\text{ $n$ odd.}
\end{cases}
\end{equation}
If $n=0$ then it is easy to see $A_0$ is well-defined since the denominator becomes $\int_{-\infty}^{\infty}\frac{d\mu(x)}{c+x^2}$ which is nonzero.
Now consider the case when $n$ is even, $n\geq 2$. 
Suppose there exists $n\geq 2$ such that  $\int_{-\infty}^{\infty}P_{n}(x)\frac{
d\mu(x)}{c+x^2}=0$. We claim that this implies
\begin{equation}\label{eq:induction}
\int_{-\infty}^{\infty}x^kP_n(x)\frac{
d\mu(x)}{c+x^2}=0
\end{equation} for all $k$ even, $k \leq n$. First, note that by the orthogonality of the $P_n(x)$,
\begin{align*}
   0&= \int_{-\infty}^{\infty}P_n(x)d\mu(x)\\
   &=\int_{-\infty}^{\infty}(c+x^2)P_n(x)\frac{
d\mu(x)}{c+x^2}\\
    &=c\int_{-\infty}^{\infty}P_n(x)\frac{
d\mu(x)}{c+x^2}+\int_{-\infty}^{\infty}x^2P_n(x)\frac{
d\mu(x)}{c+x^2}.
\end{align*}

Hence, by assumption, we have 
\[\int_{-\infty}^{\infty}x^2P_n(x)\frac{
d\mu(x)}{c+x^2}=0.\]
Now let $2 \leq k <n$, and assume that $\int_{-\infty}^{\infty}x^mP_n(x)\frac{
d\mu(x)}{c+x^2}=0$ for all $m$ even, $m\leq k$. Then
\begin{align*}
\int_{-\infty}^{\infty}(c+x^2)P_kP_n(x)\frac{
d\mu(x)}{c+x^2}&=\int_{-\infty}^{\infty}P_kP_n(x)d\mu(x)\\
&=0.
\end{align*} by the orthogonality of the $P_n(x)$. But note that since $k$ is even, $P_k(x)=x^k+a_{k,k-2}x^{k-2}+\dots+a_{k,2}x^2+a_{k,0}$, hence \[(c+x^2)P_k(x)=x^{k+2}+b_{k,k}x^{k}+\dots+b_{k,2}x^2+b_{k,0}.\] Thus,
\begin{align*}
\int_{-\infty}^{\infty}(c+x^2)P_kP_n(x)\frac{
d\mu(x)}{c+x^2}&=\int_{-\infty}^{\infty}\left(x^{k+2}+b_{k}x^{k}+\dots+b_0\right)P_n(x)\frac{
d\mu(x)}{c+x^2}\\
&=\int_{-\infty}^{\infty}x^{k+2}P_n(x)\frac{
d\mu(x)}{c+x^2}
\end{align*} 
and hence 
\[
\int_{-\infty}^{\infty}x^{k+2}P_n(x)\frac{
d\mu(x)}{c+x^2}=0.
\]Thus, by induction, \eqref{eq:induction} holds for all $k$ even, $k\leq n$. In particular,
\[
\int_{-\infty}^{\infty}P_n^2(x)\frac{
d\mu(x)}{c+x^2}=0.
\]Note that the integrand is non-negative for all $x\in$ $(-\infty, \infty)$, thus we must have that $P_n^2(x)\equiv 0$ $\mu$-almost everywhere, which is a contradiction.

We now consider the case $n$ is odd and suppose there exists $n\geq 3$ such that $\int_{-\infty}^{\infty}R_3(x;c)P_n(x)\frac{d\mu(x)}{c+x^2}=0$. Notice that
\begin{align*}
    \int_{-\infty}^{\infty} R_3(x;c)P_n(x)\frac{d\mu(x)}{c+x^2}&=\int_{-\infty}^{\infty}P_3(x)P_n(x)\frac{d\mu(x)}{c+x^2}+c\int_{-\infty}^{\infty} P_1(x)P_n(x) \frac{d\mu(x)}{c+x^2}\\
    &=0.
\end{align*} Hence
\begin{equation}\label{eq:odd1}
c\int_{-\infty}^{\infty} P_1(x)P_n(x) \frac{d\mu(x)}{c+x^2} = -\int_{-\infty}^{\infty}P_3(x)P_n(x)\frac{d\mu(x)}{c+x^2}.
\end{equation}Also, by the orthogonality of the $P_n(x)$,
\begin{align*}
   0&= \int_{-\infty}^{\infty} P_1(x)P_n(x) d\mu(x)\\
   &=\int_{-\infty}^{\infty}(c+x^2)P_1(x)P_n(x)\frac{d\mu(x)}{c+x^2}\\
   &=c\int_{-\infty}^{\infty} P_1(x)P_n(x) \frac{d\mu(x)}{c+x^2}+\int_{-\infty}^{\infty}x^2P_1(x)P_n(x)\frac{d\mu(x)}{c+x^2}\\
   &=c\int_{-\infty}^{\infty} P_1(x)P_n(x) \frac{d\mu(x)}{c+x^2}+\int x^3P_n(x)\frac{d\mu(x)}{c+x^2},
\end{align*} which implies
\begin{equation}\label{eq:odd2}
c\int_{-\infty}^{\infty} P_1(x)P_n(x)\frac{d\mu(x)}{c+x^2}=-\int_{-\infty}^{\infty}x^3P_n(x)\frac{d\mu(x)}{c+x^2}.
\end{equation}
Thus equations \eqref{eq:odd1} and \eqref{eq:odd2} give
\[
\int_{-\infty}^{\infty} P_3(x)P_n(x)\frac{d\mu(x)}{c+x^2}=\int_{-\infty}^{\infty}x^3P_n(x)\frac{d\mu(x)}{c+x^2}.
\]But recall from Lemma \ref{prop:p3} that $P_3(x)=x^3+ax$ where $a\neq 0$, hence 
\begin{align*}
\int_{-\infty}^{\infty} (x^3+ax)P_n(x)\frac{d\mu(x)}{c+x^2}&=\int_{-\infty}^{\infty} x^3P_n(x)\frac{d\mu(x)}{c+x^2}
\end{align*} which implies that 
\begin{equation*}
  \int_{-\infty}^{\infty} xP_n(x)\frac{d\mu(x)}{c+x^2}=0.
\end{equation*}

Now we proceed by induction. Let $k$ be odd, $1\leq k <n$ and suppose that $\int_{-\infty}^{\infty}x^mP_n(x)\frac{d\mu(x)}{c+x^2}=0$ for all $m$ odd, $1\leq m\leq k$. Then 
\[
\int_{-\infty}^{\infty}(c+x^2)P_k(x)P_n(x)\frac{d\mu(x)}{c+x^2}=0
\] by the orthogonality of the $P_n(x)$. Notice that $(c+x^2)P_k(x)$ is an odd, monic polynomial of degree $k+2$, thus 
\begin{align*}
    0&=\int_{-\infty}^{\infty} (c+x^2)P_k(x)P_n(x)\frac{d\mu(x)}{c+x^2}\\
    &=\int_{-\infty}^{\infty}x^{k+2}P_n(x)\frac{d\mu(x)}{c+x^2}
\end{align*} hence $\int_{-\infty}^{\infty} x^k P_n(x)\frac{d\mu(x)}{c+x^2} =0$ for all odd $k \leq n$. As before, this implies 
\[
\int_{-\infty}^{\infty} P^2_n(x)\frac{d\mu(x)}{c+x^2} =0
\] which is a contradiction.

\end{proof}
It turns out, even with the condition $A_1\neq -\frac{[P_3,x]_{\phi}}{[P_1,x]_{\phi}}$, the family $\{R_n(x;c)\}_{n\in \mathbb{N}\setminus\{1\}}$ is orthogonal. To prove this, we will need the following lemma.
\begin{lemma}\label{lemma:ratio} Let $\{P_n(x)\}_{n=0}^{\infty}$ be a monic OPS with respect to $\mu(x)$ and let  $c>0$. Then if $k \leq n$,
\begin{equation}\label{eq:ratio}
\int_{-\infty}^{\infty} x^kP_{n}(x)\frac{d\mu(x)}{c+x^2}=\begin{cases}\displaystyle (-c)^{\frac{k}{2}}\int_{-\infty}^{\infty}P_n(x)\frac{d\mu(x)}{c+x^2}& \text{ for $k$ and $n$ even,}\\
\displaystyle (-c)^{\frac{k-1}{2}}\int_{-\infty}^{\infty} xP_n(x)\frac{d\mu(x)}{c+x^2} &\text{ for $k$ and $n$ odd}.
\end{cases}
\end{equation} 
\end{lemma}
\begin{proof} Consider the case $k$ and $n$ are even. Let $k=2$. Then
\begin{align*}
    \int_{-\infty}^{\infty}(c+x^2)P_n(x)\frac{d\mu(x)}{c+x^2}&=\int_{-\infty}^{\infty}P_n(x)d\mu(x)\\
    c\int_{-\infty}^{\infty}P_n(x)\frac{d\mu(x)}{c+x^2}+\int_{-\infty}^{\infty} x^2 P_n(x)\frac{d\mu(x)}{c+x^2}&=0\\
   \int_{-\infty}^{\infty} x^2P_n(x)\frac{d\mu(x)}{c+x^2}&=-c\int_{-\infty}^{\infty}P_n(x)\frac{d\mu(x)}{c+x^2}
    \end{align*} for $n\geq 2$ by the orthogonality of the $P_n(x)$. Now, assume \eqref{eq:ratio} holds for all even $k \leq n-2$. Then 
    \begin{align*}
        \int_{-\infty}^{\infty} (c+x^2)x^k P_n(x)\frac{d\mu(x)}{c+x^2}&=\int_{-\infty}^{\infty} x^k P_n(x) d\mu(x)\\
        c\int_{-\infty}^{\infty} x^k P_n(x) \frac{d\mu(x)}{c+x^2}+\int_{-\infty}^{\infty} x^{k+2}P_n(x)\frac{d\mu(x)}{c+x^2}&=0 \\\int_{-\infty}^{\infty} x^{k+2}P_n(x) \frac{d\mu(x)}{c+x^2}
       &=(-c)^{\frac{k+2}{2}}\int_{-\infty}^{\infty}P_n(x)\frac{d\mu(x)}{c+x^2},
    \end{align*}where the last equality follows by assumption. Thus, by induction, \eqref{eq:ratio} holds for all $k\leq n$, $k$ and $n$ even.
   The case for $k$ and $n$ odd follows similarly.
\end{proof}

\begin{theorem}\label{thrm:orthog_1} If $A_1=c\neq -\frac{[P_3(x),x]_{\phi}}{[P_{1}(x),x]_{\phi}}$, then $\{R_n(x;c)\}_{n\in \mathbb{N}\setminus\{1\}}$ is  an orthogonal sequence with respect to $\frac{d\mu(x)}{c+x^2}$, which does not contain a degree 1 polynomial. i.e. 
\[
\int_{-\infty}^{\infty} R_n(x;c)R_m(x;c)\frac{d\mu(x)}{c+x^2}=\mathcal{K}_n\delta_{n,m}, \quad n,m =0,2,3,4,5,\dots
\]for some constants $\mathcal{K}_n$.
\end{theorem}

\begin{proof} Fix $c>0$ and denote $R_n(x;c)=R_n(x)$. First, note that by Theorem \ref{thrm:an_1}, $R_n(x)$ exists for all $n\in \mathbb{N}\setminus\{1\}$. Now, consider the case that $n$ and $m$ are both even $n,m \geq 2$, and let $G_m(x)=R_m(x)-R_m(i\sqrt{c})$. Then $G_m(i\sqrt{c})=0$ and $G_m(-i\sqrt{c})=R_m(-i\sqrt{c})-R_m(i\sqrt{c})=0$ since $R_m(x)$ is even. Thus, $G_m(x)=(c+x^2)g_m(x)$ where $\deg g_m(x)=m-2$. Then 
\begin{align*}
   \int_{-\infty}^{\infty} G_m(x)R_n(x)\frac{d\mu(x)}{c+x^2}&=\int_{-\infty}^{\infty} g_m(x)R_n(x)d\mu(x)\\
    &=\int_{-\infty}^{\infty}  g_m(x)P_{n}(x)d\mu(x)+A_{n-2}\int_{-\infty}^{\infty} g_m(x)P_{n-2}(x)d\mu(x)\\
    &=0 \text{ for $m<n$.}
\end{align*} But,
\begin{align*}
    \int_{-\infty}^{\infty} G_m(x)R_n(x)\frac{d\mu(x)}{c+x^2}&=\int_{-\infty}^{\infty}  R_m(x)R_n(x)\frac{d\mu(x)}{c+x^2}-R_m(i\sqrt{c})\int_{-\infty}^{\infty}  R_n(x)\frac{d\mu(x)}{c+x^2}\\
    &=\int_{-\infty}^{\infty}  R_m(x)R_n(x)\frac{d\mu(x)}{c+x^2} \end{align*}
according to \eqref{cond:1}, hence $\int_{-\infty}^{\infty}  R_m(x)R_n(x)\frac{d\mu(x)}{c+x^2} =0$ for $n,m \geq 2$ even,  $m \neq n$.

Now let $n$ and $m$ both be odd. If $n=3$, the by definition of $\{R_n(x)\}_{n\in\mathbb{N}\setminus\{1\}}$, we have that $\int_{-\infty}^{\infty}  R_3(x)R_m(x) \frac{d\mu(x)}{c+x^2}=0$ for all $m\geq 5$. If $n,m \geq 5$, then define $L_m(x):=\frac{R_m(i\sqrt{c})}{ai\sqrt{c}}\left(x^3+(a+c)x\right)$ 
where $a$ is the nonzero coefficient of $x$ in $P_3(x)=x^3+ax$. Additionally, let $H_n(x):=R_n(x)-L_n(x)$. Then one can check that $H_n(i\sqrt{c})=H_n(-i\sqrt{c})=0$ and thus $H_n(x)=(c+x^2)h_n(x)$ where $\deg h_n(x)=n-2$. Thus,
\begin{align*}
    \int_{-\infty}^{\infty}  H_m(x)R_n(x)\frac{d\mu(x)}{c+x^2}&=\int_{-\infty}^{\infty} h_m(x)R_n(x)d\mu(x)\\
    &=\int_{-\infty}^{\infty}  h_m(x)P_{n}(x)d\mu(x)+A_{n-2}\int_{-\infty}^{\infty} h_m(x)P_{n-2}(x)d\mu(x)\\
    &=0 \text{ for }m<n.
\end{align*}
But 
\begin{align*}
    \int_{-\infty}^{\infty}  H_m(x)R_n(x)\frac{d\mu(x)}{c+x^2}&=\int_{-\infty}^{\infty}  R_m(x)R_n(x) \frac{d\mu(x)}{c+x^2}-\int_{-\infty}^{\infty}  L_m(x)R_n(x) \frac{d\mu(x)}{c+x^2}.
\end{align*}Looking at the right-most integral, we see that
\begin{align*}
    \int_{-\infty}^{\infty}  L_m(x)R_n(x) \frac{d\mu(x)}{c+x^2}&=\frac{R_m(i\sqrt{c})}{ai\sqrt{c}}\int_{-\infty}^{\infty}(x^3+(a+c)x)R_n(x) \frac{d\mu(x)}{c+x^2}\\
    &=\frac{R_m(i\sqrt{c})}{ai\sqrt{c}}\int_{-\infty}^{\infty}R_3(x)R_n(x) \frac{d\mu(x)}{c+x^2}\\
    &=0 \text{ for $n \geq 5$},
\end{align*} where in the second to last line, we used the fact that $R_3(x)=x^3+(a+c)x$. 
Thus, $\int_{-\infty}^{\infty}  R_m(x)R_n(x) \frac{d\mu(x)}{c+x^2}=0$ for $m, n$ odd, $m \neq n$.

The case that $n$ is odd and $m$ is even or vice versa follows from the symmetry of $\mu(x)$ and thus $\int_{-\infty}^{\infty} R_n(x)R_m(x)\frac{d\mu(x)}{c+x^2}=0$ for all $n \neq m$. Now if $n=m$, then 
\[
\int_{-\infty}^{\infty} R^2_n(x)\frac{d\mu(x)}{c+x^2}
\neq 0\] since $\frac{R^2_n(x)}{c+x^2}$ is positive (except at finitely many points) on the support of $\mu(x)$, hence $\{R_n(x)\}_{n\in \mathbb{N}\setminus\{1\}}$ is orthogonal with respect to $\frac{d\mu(x)}{c+x^2}$. The fact that there is no degree 1 polynomial in the sequence is already discussed in the beginning of Section \ref{sec:order 1}.
\end{proof}

\begin{remark}We note that if one assumes $\frac{d\mu(x)}{x^2}$ has finite moments, then it is quick to check that the above theorems still hold when $c=0$. Thus, we can generate a lacunary sequence of orthogonal polynomials with respect to $\frac{d\mu(x)}{x^2}$ which is missing the degree 1 polynomial.
\end{remark}

\noindent It is well known that given a sequence of orthogonal polynomials $P_n(x)$ with respect to a symmetric measure $\mu(x)$, one may define new polynomials $M_n(x^2)=P_{2n}(x)$ and $xN_n(x^2)=P_{2n+1}(x)$, which are orthogonal with respect to a measure obtained from $\mu(x)$ (see \cite{Chihara}, Section 8). Thus, we explore this quadratic transformation for the family $\{R_n(x;c)\}_{n\in\mathbb{N}\setminus\{1\}}$.

Define 
\begin{align*}
M_n(x^2,c)&:=R_{2n}(x;c),\quad n\geq 0,\\
xN_n(x^2;c)&:=R_{2n+1}(x;c), \quad n\geq 1.
\end{align*}Then we have the following orthogonality relations for the sequences $\{M_n(x;c)\}_{n \geq 0}$ and $\{N_n(x;c)\}_{n \geq 1}$.

\begin{lemma}
Let $\{R_n(x;c)\}_{n\in\mathbb{N}\setminus\{1\}}$ be defined by equations \eqref{eq:Rn_2term}, \eqref{cond:1} and \eqref{cond:2} and assume the corresponding measure $\mu(x)$ has density $w(x)$. 
Then 
\begin{enumerate}
\item $\{M_n(x;c)\}_{n\geq 0}$ is orthogonal with respect to $\frac{w(\sqrt{x})}{(c+x)\sqrt{x}}\chi_{(0,\infty)}(x)$. 

\item $\{N_n(x;c)\}_{n\geq 1}$ is orthogonal with respect to $\frac{w(\sqrt{x})\sqrt{x}}{c+x}\chi_{(0,\infty)}(x)$. 
\end{enumerate}
where $\chi_A(x)$ denotes the indicator function on $A\subset \mathbb{R}$.
\end{lemma}
\begin{proof}
The proof follows from Theorem \ref{thrm:orthog_1} and the transformation $x^2 \rightarrow x$.

\end{proof}

\subsection{Example: Chebyshev Polynomials of the First Kind}\label{subsec:cheb1}

Let $T_n(x)$ denote the $n$-th degree monic Chebyshev polynomial of the first kind defined by 
\begin{align*}
T_0(x)&=1,\\ T_n(\cos\theta)&=\frac{1}{2^{n-1}}\cos(n\theta)\quad n\geq 1. 
\end{align*}
Recall that the Chebyshev polynomials of the first kind are orthogonal with respect to $\displaystyle \frac{dx}{\sqrt{1-x^2}}\chi_{(-1,1)}(x)$ and satisfy the following three-term recurrence relation:
\[
T_{n+1}(x)=xT_n(x)-\frac{1}{4}T_{n-1}(x), n\geq 2, \quad T_0(x)=1, T_1(x)=x.
\]

Then $-\frac{[P_3,x]_{\phi}}{[P_{1},x]_{\phi}}=\frac{3}{4}-\frac{\sqrt{2}}{2}$, thus letting $c=1$, we know by Theorem \ref{thrm:an_1} and Theorem \ref{thrm:orthog_1} that there exists a lacunary sequence of orthogonal polynomials  $\{\mathcal{T}_n(x;1)\}_{n\in\mathbb{N}\setminus\{1\}}$ satisfying \eqref{eq:Rn_2term}, \eqref{cond:1} and \eqref{cond:2}, which are orthogonal with respect to $\displaystyle\frac{dx}{\sqrt{1-x^2}(1+x^2)}$. We list the first 8
$\mathcal{T}_n(x;1)$'s below.

\begin{align*}
\mathcal{T}_0(x;1)&=1\\
\mathcal{T}_2(x;1)&=x^2+1-\sqrt{2}\\
\mathcal{T}_3(x;1)&=x^3+\frac{1}{4}x\\
\mathcal{T}_4(x;1)&=x^4-\frac{1+2\sqrt{2}}{4}x^2+\frac{\sqrt{2}-1}{4}\\
\mathcal{T}_5(x;1)&=x^5-\frac{1+\sqrt{2}}{2}x^3+\frac{3\sqrt{2}-2}{8}x\\
\mathcal{T}_6(x;1)&=x^6-\frac{3+2\sqrt{2}}{4}x^4+\frac{8\sqrt{2}-3}{16}x^2+\frac{1-\sqrt{2}}{16}\\
\mathcal{T}_7(x;1)&=x^7-\frac{2+\sqrt{2}}{2}x^5+\frac{10\sqrt{2}-1}{16}x^3+\frac{4-5\sqrt{2}}{32}x.
\end{align*}

Although we know from Theorem 
\ref{thrm:an_1} that the $A_n$'s in equation \eqref{eq:Rn_2term} exist, we can actually compute them explicitly for the $\{\mathcal{T}_n(x;1)\}_{n\in\mathbb{N}\setminus \{1\}}$.
\begin{proposition}\label{prop:An}
   Let $T_n(x)$ denote the $n$-th degree Chebyshev polynomial of the first kind. Then for $n\geq 0$, the sequence $\{\mathcal{T}_n(x,1)\}_{n \in \mathbb{N}\setminus\{1\}}$ satisfies  \[\mathcal{T}_n(x;1)=T_n(x)+A_{n-2}T_{n-2}(x),\] where $A_1=1$ and $A_n=\frac{3}{4}-\frac{\sqrt{2}}{2}$ for all $n \geq 2$ where we define $T_{n}(x):=0$ for $n< 0$.
\end{proposition}
\begin{proof}
First, let $\overline{T}_n(x)=2^{n-1}T_n(x), \, n\geq 1, $ be the non-monic $n$-th degree Chebyshev polynomial of the first kind, ($\overline{T}_0(x)=1)$. Then it is well known that
\[
\int_{-1}^1\frac{\overline{T}_n(x)-\overline{T}_n(y)}{(x-y)\sqrt{1-x^2}}dx=\pi U_{n-1}(y)
\] where $U_{n-1}(y)$ is the $(n-1)$-st degree Chebyshev polynomial of the second kind. Thus,
\begin{equation}\label{eq:chebsecond}
\int_{-1}^1\frac{\overline{T}_n(x)}{(x-y)\sqrt{1-x^2}}dx=\pi U_{n-1}(y)+\overline{T}_n(y)\int_{-1}^1\frac{1}{(x-y)\sqrt{1-x^2}}dx.
\end{equation}
Thus, by partial fraction decomposition we have
\begin{align*}
    \int_{-1}^1\frac{\overline{T}_n(x)}{(1+x^2)\sqrt{1-x^2}}dx&=\frac{1}{2i}\int_{-1}^1\frac{\overline{T}_n(x)}{(x-i)\sqrt{1-x^2}}dx-\frac{1}{2i}\int_{-1}^1\frac{\overline{T}_n(x)}{(x+i)\sqrt{1-x^2}}dx\\
    &=\frac{1}{2i}\left(\pi U_{n-1}(i)+\overline{T}_n(i)\int_{-1}^1\frac{1}{(x-i)\sqrt{1-x^2}}dx\right)-\\
    &\frac{1}{2i}\left(\pi U_{n-1}(-i)+\overline{T}_n(-i)\int_{-1}^1\frac{1}{(x+i)\sqrt{1-x^2}}dx\right)
.\end{align*}
Assume $n$ is even. Then 
\begin{align*}
    \int_{-1}^1\frac{\overline{T}_n(x)}{(1+x^2)\sqrt{1-x^2}}dx&=\frac{\pi}{i} U_{n-1}(i) +\frac{\overline{T}_n(i)}{2i}\int_{-1}^1 \frac{2i}{(1+x^2)\sqrt{1-x^2}}dx\\
    &=\frac{\pi}{i}U_{n-1}(i)+\frac{\pi \overline{T}_n(i)}{\sqrt{2}}\\
    &=\pi\left(\frac{\overline{T}_n(i)}{\sqrt{2}}-i U_{n-1}(i)\right).
    \end{align*}
    Using the fact that $\displaystyle \overline{T}_n(x)=\frac{(x-\sqrt{x^2-1})^n+(x+\sqrt{x^2-1})^n}{2}$ and\newline
    $\displaystyle U_{n-1}(x)=\frac{(x+\sqrt{x^2-1})^{n}-(x-\sqrt{x^2-1})^{n}}{2\sqrt{x^2-1}}$ we have that
    \begin{align*}
   \pi\left(\frac{\overline{T}_n(i)}{\sqrt{2}}-i U_{n-1}(i)\right)&=\pi \left( \frac{(i-i\sqrt{2})^n+(i+i\sqrt{2})^n}{2\sqrt{2}}\right)-i\left(\frac{(i+i\sqrt{2})^n-(i-i\sqrt{2})^n}{2\sqrt{2}i}\right)\\
   &=\frac{\pi}{\sqrt{2}}i^n(1-\sqrt{2})^n.
    \end{align*}

Thus, for $n\geq 2$ even, we have that
\begin{align*}
    A_n&=\frac{-\displaystyle\int_{-1}^{1}\frac{T_{n+2}(x)
}{(1+x^2)\sqrt{1-x^2}}}{ \displaystyle\int_{-1}^{1}\frac{T_{n}(x)}{(1+x^2)\sqrt{1-x^2}}}\\
&= \frac{-\frac{1}{2^{n+1}}\displaystyle \int_{-1}^{1}\frac{\overline{T}_{n+2}(x)
}{(1+x^2)\sqrt{1-x^2}}}{ \frac{1}{2^{n-1}}\displaystyle \int_{-1}^{1}\frac{\overline{T}_{n}(x)}{(1+x^2)\sqrt{1-x^2}}}\\
&=-\frac{1}{4}\frac{\frac{\pi}{\sqrt{2}}i^{n+2}(1-\sqrt{2})^{n+2}}{\frac{\pi}{\sqrt{2}}i^n(1-\sqrt{2})^n}\\
&=\frac{-i^2(1-\sqrt{2})^2}{4}\\
&=\frac{3}{4}-\frac{\sqrt{2}}{2}.
\end{align*}
The case when $n \geq 3$ is odd follows similarly. 
\end{proof}

Proposition \ref{prop:An} allows us to obtain the following 5-term recurrence relation.

\begin{theorem}
For $n=4$ and $n\geq 6$, the polynomials $\mathcal{T}_n(x;1)$ satisfy the following 5-term recurrence relation:
   \[
(1+x^2)\mathcal{T}_{n}(x;1)=\mathcal{T}_{n+2}(x;1)+\frac{3}{2}\mathcal{T}_n(x;1)+\frac{1}{16}\mathcal{T}_{n-2}(x;1).
\]
\end{theorem}

\begin{proof}Denote $\mathcal{T}_n(x;1)=\mathcal{T}_n(x)$. One may check explicitly that
\[
(1+x^2)\mathcal{T}_4(x)=\mathcal{T}_6(x)+\frac{3}{2}\mathcal{T}_4(x)+\frac{1}{16}\mathcal{T}_2(x).
\]
Now let $n\geq 5$. It follows from Proposition \ref{prop:An},
\[x\mathcal{T}_n(x)=xT_n(x)+\left(\frac{3}{4}-\frac{\sqrt{2}}{2}\right)xT_{n-2}(x),\] where $T_n(x)$ is the $n$-th degree monic Chebyshev polynomial of the first kind. Note that the monic Chebyshev polynomials satisfy
\[
xT_n(x)=T_{n+1}(x)+\frac{1}{4}T_{n-1}(x), \quad n\geq 2.
\] Thus, for $n\geq 4$ we have,
\begin{align*}
x\mathcal{T}_n(x)&=T_{n+1}(x)+\left(\frac{1}{4}+\frac{3}{4}-\frac{\sqrt{2}}{2}\right)T_{n-1}(x)+\frac{1}{4}\left(\frac{3}{4}-\frac{\sqrt{2}}{2}\right)T_{n-3}(x)
    \end{align*}
and for $n\geq 5$ we have,
\begin{align*}
x^2\mathcal{T}_n(x)=T_{n+2}(x)+\left(\frac{5}{4}-\frac{\sqrt{2}}{2}\right)T_n(x)+\left(\frac{7}{16}-\frac{\sqrt{2}}{4}\right)&T_{n-2}(x)\\
&
+\frac{1}{16}\left(\frac{3}{4}-\frac{\sqrt{2}}{2}\right)T_{n-4}(x).
\end{align*}
Hence, 
\begin{align*}
    x^2\mathcal{T}_n(x)+\mathcal{T}_n(x)&=T_{n+2}(x)+\left(\frac{9}{4}-\frac{\sqrt{2}}{2}\right)T_n(x)\\
    &+\left(\frac{19}{16}-\frac{3\sqrt{2}}{4} \right)T_{n-2}(x)+\frac{1}{16}\left(\frac{3}{4}-\frac{\sqrt{2}}{2}\right)T_{n-4}(x)\\
    &=T_{n+2}(x)+\left(\frac{3}{4}-\frac{\sqrt{2}}{2}\right)T_n(x)+\frac{3}{2}\left(T_n(x)+\left(\frac{3}{4}-\frac{\sqrt{2}}{2}\right)T_{n-2}(x)\right)\\&+\frac{1}{16}\left(T_{n-2}(x)+\left(\frac{3}{4}-\frac{\sqrt{2}}{2}\right)T_{n-4}(x)\right)\\
    &=\mathcal{T}_{n+2}(x)+\frac{3}{2}\mathcal{T}_n(x)+\frac{1}{16}\mathcal{T}_{n-2}(x),
\end{align*} where the last equality holds for $n\geq 6$. After factoring the left hand side, we arrive at the desired relation.
\end{proof}

We plot the  $\mathcal{T}_n(x;1)$ for $n=2,3,4,5$ below.

\begin{figure}[h!]
    \centering
    \includegraphics[width=0.8\textwidth]{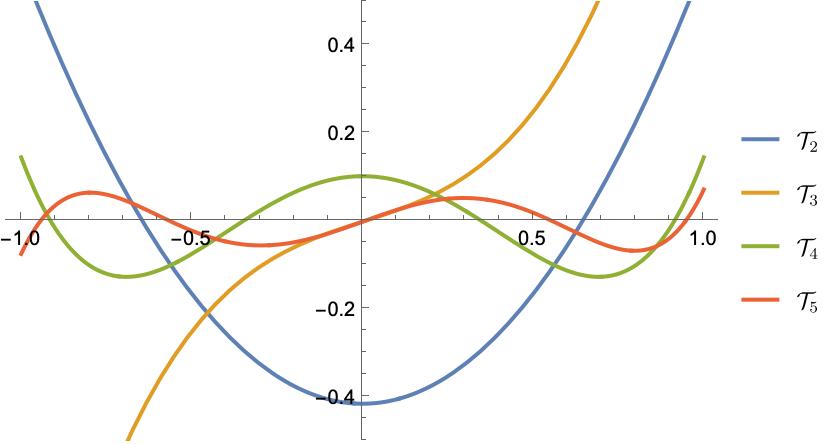}
    \caption{The $\mathcal{T}_n(x;1)$ for $n=2,3,4,5$.}
    \label{fig:plottn}
\end{figure}

\section{relation to quasi-orthogonal polynomials of order 2}\label{sec:quasi}
In the context of our particular problem, \cite{Marcellan2011} characterizes the orthogonality of quasi-orthogonal polynomial sequences $\{Q_n(x)\}_{n=0}^{\infty}$ of order 2 defined by 
\[
Q_n(x):=P_n(x)+s_nP_{n-1}(x)+t_nP_{n-2}(x), 
\] with $t_0=t_1=0,$ and $ t_n\neq 0$ for $n\geq 2$. When $\mu(x)$ is symmetric on $\mathbb{R}$, this reduces to $s_n=0$ for all $n$.
In this case, Theorem 2.2 of \cite{Marcellan2011} states that $\{Q_n(x)\}_{n=0}^{\infty}$ is a monic OPS with respect to $\frac{d\mu(x)}{c+x^2}$ if and only if there exists a complex number $c$ such that, for $n \geq 1$ 
\begin{equation}\label{eq:bn}
B_n:=\frac{1}{t_{n+1}}\left(\lambda_{n+2}-t_{n+2} \right)\left(\lambda_{n+1}+t_n-t_{n+1}\right)+t_n-\lambda_{n}=c,\end{equation} where
$\{\lambda_n\}_{n=1}^{\infty}$ is the sequence of (real) recurrence coefficients for $\{P_n(x)\}_{n=0}^{\infty}$ in  equation \eqref{eq:recrel}. We remark here that in this case, the $Q_n(x)$ are the polynomials formed via Gram–Schmidt with respect to $\frac{d\mu(x)}{c+x^2}.$

The authors also discuss the symmetric case, and show that if $\{Q_n(x)\}_{n=0}^{\infty}$ is an OPS, then the coefficients $t_n$ can be computed in terms of the coefficients $t_2$, $t_3$ and the three-term recurrence relation for $\{P_n(x)\}_{n=0}^{\infty}$. They state that if $\{Q_n(x)\}_{n=0}^{\infty}$ is an OPS, then the coefficients satisfy \eqref{eq:bn}, and the orthogonality measure is $\frac{d\mu(x)}{c+x^2}$. Below, we provide an explicit formula for the coefficients $t_n$ in terms of the polynomials $P_n(x)$ and show these coefficients always satisfy \eqref{eq:bn}.

\begin{theorem}\label{thrm:quasicoef}
    Let $\{P_n(x)\}_{n=0}^{\infty}$ be a monic OPS with respect to a symmetric measure $\mu(x)$ supported on a symmetric, infinite subset of $\mathbb{R}$. Then the quasi-orthogonal polynomials $Q_n(x)=P_n(x)+t_{n}P_{n-2}(x)$ form a monic OPS with respect to $\frac{d\mu(x)}{c+x^2}$ if and only if for all $n\geq 2$,  
\[
t_n=\begin{cases}\frac{\displaystyle -\int_{-\infty}^{\infty}P_{n}(x)\frac{
d\mu(x)}{c+x^2}}{ \displaystyle \int_{-\infty}^{\infty}P_{n-2}(x)\frac{
d\mu(x)}{c+x^2}},& \text{ $n$ even}\\
\frac{\displaystyle -\int_{-\infty}^{\infty}xP_{n}(x)\frac{
d\mu(x)}{c+x^2}}{ \displaystyle \int_{-\infty}^{\infty}xP_{n-2}(x)\frac{
d\mu(x)}{c+x^2}},&\text{ $n$ odd}
\end{cases}
\] and $t_0=t_1=0$.
\end{theorem}

\begin{proof}
Let $\langle f,g \rangle:=\int_{-\infty}^{\infty}f(x)g(x)\frac{d\mu(x)}{c+x^2}$ and assume $n$ is even. Then  
\begin{multline}\label{eq:iffproof}
\frac{1}{t_{n+1}}\left(\lambda_{n+2}-t_{n+2} \right)\left(\lambda_{n+1}+t_n-t_{n+1}\right)+t_n-\lambda_{n}=\\-\frac{\langle x,P_{n-1}\rangle}{\langle x, P_{n+1}\rangle}\left(\lambda_{n+2}+\frac{\langle 1, P_{n+2}\rangle}{\langle 1, P_{n}\rangle}\right)\left(\lambda_{n+1}-\frac{\langle 1, P_{n}\rangle}{\langle 1 ,P_{n-2}\rangle}+\frac{\langle x,P_{n+1}\rangle}{\langle x, P_{n-1}\rangle}\right)\\
-\frac{\langle 1, P_{n}\rangle}{\langle 1, P_{n-2}\rangle}-\lambda_{n}\\
= \frac{-\lambda_{n+2}\langle  \lambda_{n+1}P_{n-1},x\rangle}{\langle P_{n+1},x \rangle}-\frac{\lambda_{n+1}\langle P_{n-1},x \rangle \langle P_{n+2},1\rangle }{\langle P_{n+1},x\rangle \langle P_{n},1\rangle}+\frac{\lambda_{n+2}\langle P_{n-1},x \rangle \langle P_{n},1\rangle }{\langle  P_{n+1},x \rangle \langle P_{n-2},1\rangle}\\
+\frac{\langle P_{n-1},x\rangle \langle P_{n+2},1\rangle }{\langle P_{n+1},x \rangle \langle P_{n-2},1 \rangle}-\frac{\langle P_{n+2},1 \rangle}{\langle P_n,1 \rangle}-\frac{\langle P_{n},1\rangle}{\langle  P_{n-2},1 \rangle}-\lambda_{n+2}-\lambda_n,
\end{multline}
Using the fact that $\lambda_{n+1}P_{n-1}(x)=xP_n(x)-P_{n+1}(x)$ and $\lambda_{n+2}P_n(x)=xP_{n+1}(x)-P_{n+2}(x)$ gives \eqref{eq:iffproof} becomes
\begin{equation}\label{eq:iff2}
    \frac{-\lambda_{n+2}\langle P_n, x^2\rangle }{\langle P_{n+1},x\rangle }-\frac{\langle P_{n+2},1\rangle \langle P_n,x^2\rangle}{\langle P_{n+1},x\rangle \langle P_n, 1, \rangle }+\frac{\langle P_{n-1},x\rangle}{\langle P_{n-2},1\rangle}-\frac{\langle P_{n},1\rangle }{\langle P_{n-2},1\rangle}-\lambda_n.
\end{equation}Using Lemma \ref{lemma:ratio}, we know that for $n\geq 2$, $\langle P_n, x^2\rangle = -c\langle P_n, 1\rangle$, thus \eqref{eq:iff2} is equal to 
\begin{align*}
\frac{1}{t_{n+1}}\left(\lambda_{n+2}-t_{n+2} \right)&\left(\lambda_{n+1}+t_n-t_{n+1}\right)+t_n-\lambda_{n}=\\
&=\frac{c\langle \lambda_{n+2}P_n+P_{n+2},1\rangle}{\langle P_{n+1},x\rangle}+\frac{\langle xP_{n-1}-P_n,1\rangle}{\langle P_{n-2},1\rangle}-\lambda_n
\\&=\frac{c\langle P_{n+1},x\rangle}{\langle P_{n+1},x\rangle}+\frac{\langle \lambda_nP_{n-2},1\rangle}{\langle P_{n-2},1\rangle}-\lambda_n\\
&=c. \hspace{3.4in}
\end{align*}
The case for $n \geq 1$ odd follows similarly. Hence, by Theorem 2.2 of \cite{Marcellan2011}, $\{Q_n(x)\}_{n=0}^{\infty}$ is an OPS with respect to $\frac{d\mu(x)}{c+x^2}$.

It is clear that if $\{Q_n(x)\}_{n=0}^{\infty}$ is an OPS with respect to $\frac{d\mu(x)}{c+x^2}$, then for $n \geq 1$,
\begin{equation*}
    \int_{-\infty}^{\infty}Q_n(x)\frac{d\mu(x)}{c+x^2}=0
\end{equation*} and for $n \geq 2,$
  \begin{equation*}
    \int_{-\infty}^{\infty}xQ_n(x)\frac{d\mu(x)}{c+x^2}=0
\end{equation*} thus, the coefficients must take the form in Theorem \ref{thrm:quasicoef}.
   \end{proof}

Notice from \eqref{eq:an_1} that for $n\geq 0$ even, $t_{n+2}=A_n$. Using Lemma \ref{lemma:ratio}, one can check that $t_{n+2}=A_n$ for $n\geq 3$ odd. Thus, defining polynomials $Q_n(x)$ as in Theorem \ref{thrm:quasicoef} results in an OPS with respect to $\frac{d\mu(x)}{c+x^2}$ such that $Q_n(x)=R_n(x)$ except for $n = 1, 3$. In other words, the lacunary sequence $\{R_n(x)\}_{n\in \mathbb{N}\setminus \{1\}}$ and the OPS $\{Q_n(x)\}_{n=0}^{\infty}$ overlap at all except two polynomials. Similar behavior is also seen in \cite{MJ26} where a regularized linear functional ${\bf w}$ and corresponding lacunary sequence is studied. In that case, the lacunary sequence overlaps with the OPS with respect $\bf{w}$ for all even degrees. This is not surprising since the coefficients defining the even degree polynomials are given by the condition $\langle P_n(x),1\rangle_{\bf w}=0$, which is equivalent to our condition \eqref{cond:1}.

\begin{example}Let $T_n(x)$ be the monic $n$-th degree Chebyshev polynomial of the first kind. We saw in Proposition \ref{prop:An} that $A_n=\frac{3}{4}-\frac{\sqrt{2}}{2}$. Thus, letting $t_n=\frac{3}{4}-\frac{\sqrt{2}}{2}$ for $n \geq 2$, Theorem \ref{thrm:quasicoef} states that defining $Q_n(x):=T_n(x)+t_nT_{n-2}(x)$ results in a monic OPS. Even further,  using Theorem 2.2 of \cite{Marcellan2011}, one can check that $\{Q_n(x)\}_{n=0}^{\infty}$ is an OPS with respect to $\frac{dx}{(1+x^2)\sqrt{1-x^2}}$.
\end{example}

We emphasize that the analysis in the previous section does not guarantee the orthogonality of the sequence $\{R_n(x)\}_{n\in \mathbb{N}\setminus \{1\}}$ (in particular, with $R_3(x)$), which is why the results of Section \ref{sec:order 1} are still new.

\section{a limiting case}\label{section:limit}
We now consider when $\{P_n(x)\}_{n=0}^{\infty}$ is one of the following families of classical monic orthogonal polynomials: the Chebyshev polynomials of the first kind $T_n(x)$, the Chebyshev polynomials of the second kind $U_n(x)$, the Legendre polynomials $L_n(x)$, or the Hermite polynomials $H_n(x)$. Our orthogonality measure under consideration will then be of the form $\frac{d\mu(x)}{c+x^2}=\frac{w(x)}{c+x^2}\chi_{(-b,b)}(x)\, dx$ where $w(x)$ is the weight function of one of the aforementioned classical OPSs, supported on $(-b,b)$ (note that $b$ is infinite in the case of the Hermite polynomials). 
Let $\{Q_n(x;c)\}_{n=0}^{\infty}$ be the quasi-orthogonal polynomial sequence defined by
\begin{equation}\label{eq:Qn}Q_n(x;c)=P_n(x)+a_{n,n-2}(c)P_{n-2}(x).
\end{equation}

Then we know by Theorem \ref{thrm:quasicoef} what form the coefficients $a_{n,n-2}$ must take so that the $Q_n(x;c)$ form an OPS with respect to $\frac{w(x)}{c+x^2}\chi_{(-b,b)}(x)$. In this section, we investigate our second perturbation which is the limiting behavior of these polynomials as $c\rightarrow 0$.

\begin{theorem}\label{thrm:limit_an}Let $\{P_n(x)\}_{n=0}^{\infty}$ be the monic Chebyshev, Legendre or Hermite polynomials and let $w(x)$ be the corresponding orthogonality weight. Also, let $\{Q_n(x;c)\}_{n=0}^{\infty}$ be the sequence of monic orthogonal polynomials with respect to $\frac{w(x)}{c+x^2}$ on $(-b, b)$ defined by \eqref{eq:Qn}. Then, for $n\geq 2$,
\[\lim\limits_{c\to 0}Q_n(x;c)=P_n(x)-\frac{P_n(0)}{P_{n-2}(0)}P_{n-2}(x).\]

\end{theorem}
\begin{proof}
By Theorem \ref{thrm:quasicoef}, we know that if $n \geq 2$ is even then
\[
a_{n,n-2}(c)=\frac{-\int_{-b}^{b}P_n(x)\frac{w(x)}{c+x^2}\,dx}{\int_{-b}^{b}P_{n-2}(x)\frac{w(x)}{c+x^2}\,dx}.
\]
Let $G_n(x)=P_n(x)-P_n(i\sqrt{c})$. Then since $P_n(x)$ is an even function, we have that $G(i\sqrt{c})=G(-i\sqrt{c})=0$. Therefore, $G_n(x)=(c+x^2)g_{n-2}(x)$ for a polynomial $g_{n-2}(x)$ of degree $n-2$. Hence, 
\begin{align*}
    \int_{-b}^{b}P_n(x)\frac{w(x)}{c+x^2}\,dx&= \int_{-b}^{b}g_{n-2}(x)w(x)\,dx+ P_n(i\sqrt{c})\int_{-b}^{b}\frac{w(x)}{c+x^2}\,dx\\
    &=\int_{-b}^{b}g_{n-2}(x)w(x)\,dx+ P_n(i\sqrt{c})\int_{-b}^b\frac{w(x)}{c+x^2}\,dx.
\end{align*} Notice that since $P_n(x)$ has rational coefficients, $g_{n-2}(x)$ is a polynomial with coefficients which are polynomials in $c$. Hence, $\int_{-b}^{b}g_{n-2}(x)w(x)\,dx =\Omega(n,c)$, where $\Omega(n,c)$ is a polynomial in $c$. Also, for the weight functions of the  Chebyshev, Legendre and Hermite polynomials, we have
\[
\int_{-b}^b\frac{w(x)}{c+x^2}=\frac{f_w(c)}{\sqrt{c}}
\]where $f_w(c)$ is function depending on the weight $w(x)$ such that $\lim_{c\rightarrow 0}f_w(c)$ is finite. Therefore,

\begin{align*}
   \lim\limits_{c\to 0} a_{n,n-2}(c)&=\displaystyle\frac{-\int_{-b}^{b}P_n(x)\frac{w(x)}{c+x^2}\,dx}{\int_{-b}^{b}P_{n-2}(x)\frac{w(x)}{c+x^2}\,dx}\\
    &=\lim\limits_{c\to 0} \displaystyle\frac{-\left(\Omega(n,c)+P_n(i\sqrt{c})\frac{f_w(c)}{\sqrt{c}}\right)}{\Omega(n-2,c)+P_{n-2}(i\sqrt{c})\frac{f_w(c)}{\sqrt{c}}}\\
    &=\lim\limits_{c\to 0} \displaystyle\frac{-\left(\sqrt{c}\Omega(n,c)+P_n(i\sqrt{c})f_w(c)\right)}{\sqrt{c}\Omega(n-2,c)+P_{n-2}(i\sqrt{c})f_w(c)}\\
    &=-\frac{P_{n}(0)}{P_{n-2}(0)}.
\end{align*}
Now let $n$ be odd. By Theorem \ref{thrm:quasicoef}, we know that if $n \geq 3$, then
\[
a_{n,n-2}(c)=\frac{-\int_{-b}^{b}xP_n(x)\frac{w(x)}{c+x^2}\,dx}{\int_{-b}^{b}xP_{n-2}(x)\frac{w(x)}{c+x^2}\,dx}.
\]
Define $G_n(x)=xP_n(x)-i\sqrt{c}P_n(i\sqrt{c})$. Then $G_n(x)=g_{n-1}(x)(c+x^2)$ where $g_{n-1}(x)$ is a polynomial of degree $n-1$. 

Therefore, 
\begin{align*}
\int_{-b}^{b}xP_n(x)\frac{w(x)}{c+x^2}\, dx&=\int_{-b}^{b}g_{n-1}(x)w(x)\, dx+i\sqrt{c}P_n(i\sqrt{c})\int_{-b}^{b}\frac{w(x)}{c+x^2}\, dx\\
&=\int_{-b}^{b}g_{n-1}(x)w(x)\, dx+i\sqrt{c} P_n(i\sqrt{c})\frac{f_w(c)}{\sqrt{c}}\\
&=\int_{-b}^{b}g_{n-1}(x)w(x)\, dx+i P_n(i\sqrt{c})f_w(c)
\end{align*} where $f_w(c)$ is the same function depending on the weight $w(x)$ as in the even case.

First, consider the right-most term $i P_n(i\sqrt{c})f_w(c)$. Since $n$ is odd and greater than or equal to 3, $P_n(x)$ is divisible by $x^2$, hence $P(i\sqrt{c})=c\cdot p_{n-2}(i\sqrt{c})$ for some polynomial $p_{n-2}(x)$ of degree $n-2$. Thus, letting $c$ tend to zero, we see that
\[
\lim\limits_{c\to 0}iP_n(i\sqrt{c})f_w(c)=0.
\]

Now, let's consider $\int_{-b}^{b}g_{n-1}(x)w(x)\, dx$. Then
\begin{align*}
   \lim\limits_{c\to 0}\int_{-b}^{b}g_{n-1}(x)w(x)\, dx&=\lim\limits_{c\to 0}\int_{-b}^{b}\frac{xP_n(x)-i\sqrt{c}P_n(i\sqrt{c})}{c+x^2}w(x)\, dx \\
   &=\int_{-b}^{b}\lim\limits_{c\to 0}\frac{xP_n(x)-i\sqrt{c}P_n(i\sqrt{c})}{c+x^2}w(x)\, dx\\
   &=\int_{-b}^{b}\frac{P_n(x)}{x}w(x)\, dx.
   \end{align*}
where we can exchange the limit and integral since $g_{n-1}(x)$ is a polynomial with coefficients which are polynomials in $c$.

Thus,
\begin{align*}
    \lim\limits_{c\to 0}a_{n,n-2}(c)&=\lim\limits_{c\to 0}\frac{-\displaystyle\int_{-b}^{b}xP_n(x)\frac{w(x)}{c+x^2}\,dx}{\displaystyle\int_{-b}^{b}xP_{n-2}(x)\frac{w(x)}{c+x^2}\,dx}\\
    &=\frac{-\displaystyle\int_{-b}^{b}\frac{P_n(x)}{x}w(x)\, dx}{\displaystyle\int_{-b}^{b}\frac{P_{n-2}(x)}{x}w(x)\, dx}.
\end{align*}
Using the three-term recurrence for the $\{P_n(x)\}_{n=0}^{\infty}$, we see
\begin{align*}
   \int_{-b}^{b}\frac{P_n(x)}{x}w(x)\, dx &=\int_{-b}^{b}\left(P_{n-1}(x)-\left(\frac{P_n(0)}{P_{n-2}(0)}\right)\left(\frac{P_{n-2}(x)}{x}\right)\right)w(x)\, dx\\
   &=-\frac{P_n(0)}{P_{n-2}(0)}\int_{-b}^{b}\frac{P_{n-2}(x)}{x}w(x)\, dx,
\end{align*}
which implies that, as before, $\lim\limits_{c\to 0}a_{n,n-2}(c)=-\frac{P_n(0)}{P_{n-2}(0)}$. 
Therefore, for $n\geq 2$, \[\lim\limits_{c\to 0}Q_n(x;c)=P_{n}(x)+\lim\limits_{c\to 0}a_{n,n-2}(c)P_{n-2}(x)=P_n(x)-\frac{P_n(0)}{P_{n-2}(0)}P_{n-2}(x).\]
\end{proof}
\noindent We will now define for $n\geq 2$, \[Q_n(x):=\lim_{c\to 0}Q_n(x;c),\] and $Q_1(x):=P_1(x)=x$. Using Theorem \ref{thrm:limit_an}, we can show that the polynomials $\{Q_n(x)\}_{n=1}^{\infty}$ form an orthogonal sequence that omits a degree zero polynomial.
\begin{theorem}
Let $\{P_n(x)\}_{n=0}^{\infty}$ be the Chebyshev, Legendre or Hermite polynomials and let $w(x)$ be the corresponding orthogonality weight. Also, let $Q_n(x)=\lim\limits_{c\to 0} Q_n(x;c)$. Then the sequence $\{Q_n(x)\}_{n=1}^{\infty}$ is a lacunary sequence of polynomials of codimension 1, which are orthogonal with respect to $\frac{w(x)}{x^2}$.
\end{theorem}
\begin{proof}
First, by Theorem \ref{thrm:limit_an}  and the recurrence relation for the $\{P_n(x)\}_{n=0}^{\infty}$, we can see that $Q_n(x)=xP_{n-1}(x)$ for $n\geq 2$. Thus, for $n,m \geq 2$ such that $n\neq m$,
\begin{align*}
\int_{-b}^{b}P_n(x)P_m(x)\frac{w(x)}{x^2}\,dx&=\int_{-b}^{b}xP_{n-1}(x)xP_{m-1}(x)\frac{w(x)}{x^2}\,dx\\
&=\int_{-\infty}^{\infty}P_{n-1}(x)P_{m-1}(x)w(x)\,dx\\
&=0.
\end{align*} 
Also, for $n\geq 2$,
\begin{align*}
    \int_{-b}^{b}Q_1(x)Q_n(x)\frac{w(x)}{x^2}\,dx&=\int_{-b}^{b}P_{n-1}(x)w(x)\,dx\\
    &=0.
\end{align*} 
If $n=m$ then,
\[
\int_{-b}^{b}Q^2_n(x)\frac{w(x)}{x^2}\,dx=\int_{-b}^{b}P^2_{n-1}(x)w(x)\,dx\neq 0.
\]
However, consider $Q_0(x)=P_0(x)=1$, and $Q_2(x)=P_2(x)+P_2(0)P_0(x)=x^2$. 
Thus,
\begin{align*}\int_{-\infty}^{\infty}
Q_0(x)Q_2(x)\frac{w(x)}{x^2}dx&=\int_{-\infty}^{\infty}w(x)dx\\
&\neq 0.
\end{align*} Hence, we see that $\{Q_n(x)\}_{n=1}^{\infty}$ a lacunary sequence of orthogonal polynomials missing a degree 0 polynomial.

We will now refer to the polynomials $\{Q_n(x)\}_{n=1}^{\infty}$ as the lacunary sequence of orthogonal polynomials associated to $\{P_n(x)\}_{n=0}^{\infty}$ where $\{P_n(x)\}_{n=0}^{\infty}$ is the underlying sequence in \eqref{eq:Qn}.

\end{proof}

Theorem \ref{thrm:limit_an} allows us to show that the $\{Q_n(x)\}_{n=1}^{\infty}$ are solutions to second order differential equations.
\begin{corollary}
The lacunary sequence of orthogonal polynomials $\{Q_n(x)\}_{n=1}^{\infty}$ associated to $\{P_n(x)\}_{n=0}^{\infty}$ satisfy the following second-order differential equation:
    \[
    y''-A(x)y'+B(n,x)y=0.
    \]
    where 
    \begin{itemize}
        \item If $P_n(x)=H_n(x)$: $A(x)=\displaystyle \frac{2(1+x^2)}{x}$ and $B(x,n)=\displaystyle \frac{2+2nx^2}{x^2}$.
        \item If $P_n(x)= T_n(x)$: $A(x)=\displaystyle \frac{2-x^2}{x(1-x^2)}$ and $B(x,n)=\displaystyle\frac{2+n(n-2)x^2}{x^2(1-x^2)}$.
         \item If $P_n(x)= U_n(x)$: $A(x)=\displaystyle\frac{2+x^2}{x(1-x^2)}$ and $B(x,n)=\displaystyle\frac{2+n^2x^2}{x^2(1-x^2)}$.
        \item  If $P_n(x)=L_n(x)$: $A(x)=\displaystyle \frac{2}{x(1-x^2)}$ and $B(x,n)=\displaystyle \frac{2+n(n-1)x^2}{x^2(1-x^2)}$.
  
    \end{itemize}
\end{corollary}
\begin{proof}
Since each family listed above satisfies a linear second-order differential equation, the proof follows from the substitution $P_{n-1}(x)=\frac{Q_n(x)}{x}$ into the respective differential equations.
\end{proof}

A classification of XOPs omitting only the degree 0 polynomial has been given in \cite{Milson2009}, where it is shown that if the eigenfunctions of a polynomial Sturm–Liouville problem form a 1-OPS, then they must be (up to an affine-transformation), either the $X_1-$Laguerre or $X_1$-Jacobi polynomials. Our polynomials under study do not contradict this statement since they are not complete in $L^2(\frac{w(x)}{x^2},(-b,b))$ thus the $\{Q_n(x)\}_{n=1}^{\infty}$ do not form a 1-OPS.

However, if we consider the subspace of $L^2(\frac{w(x)}{x^2},(-b,b))$ such  that $f(-x)=f(x)$, then we have the following.

\begin{theorem}\label{thrm:generaldense}
  Let $\{Q_n(x)\}_{n=1}^{\infty}$ be the lacunary polynomials associated to $\{P_n(x)\}_{n=0}^{\infty}$ where $\{P_n(x)\}_{n=0}^{\infty}$ is the sequence of monic Hermite, Chebyshev or Legendre Polynomials.
  Let $\mathcal{S}=\{f\in L^2(\frac{w(x)}{x^2},(-b,b)):  f(-x)=f(x)\}$. Then the polynomials $\{Q_{2k}(x)\}_{k=1}^{\infty}$ are complete in $\mathcal{S}$.
\end{theorem}
We will first use the following lemma, which we include the proof of for the reader's convenience.
\begin{lemma}\label{lemma:complete}
Let $\{P_n(x)\}_{n=0}^{\infty}$ be either the Hermite, Chebyshev, or Legendre polynomials with corresponding weight function $w(x)$ on $(-b,b)$. Then the odd degree polynomials $\{P_{2k+1}(x)\}_{n=0}^{\infty}$ form a basis for the subspace of odd functions in $L^2\left(w(x),(-b,b) \right)$.
\end{lemma}
\begin{proof}
Recall that $\{P_n(x)\}_{n=0}^{\infty}$ is a complete set in $L^2\left(w(x),(-b,b)\right)$, thus if $f(x)$ is an odd function in $ L^2\left(w(x),(-b,b)\right)$, then $f(x)=\sum_{k=0}^{\infty}a_kP_k(x)$ where the sum converges in the $L^2$ sense. Since $f(x)$ is an odd function,  the Fourier coefficients $a_k$ are all zero when $k$ is even. Hence, the $\{P_{2k+1}(x)\}_{k=0}^{\infty}$ approximate $f(x)$ in $L^2\left(w(x),(-b,b)\right)$.

\end{proof}

\noindent \textit{Proof of Theorem \ref{thrm:generaldense}.}
Suppose $f(x)\in \mathcal{S}$. Then
\[
\int_{-b}^{b}|f(x)|^2\frac{w(x)}{x^2}\, dx <\infty.
\] Hence,
\[
\int_{-b}^{b}\left|\frac{f(x)}{x}\right|^2w(x)\, dx<\infty.
\]Fix $\epsilon>0$. Since $f(x)$ is even, $\frac{f(x)}{x}$ is an odd function in $L^2\left(w(x),(-\infty, \infty)\right)$, thus by Lemma \ref{lemma:complete}, there exists an integer $N\geq 0$ and constants $a_k$ such that
\[
\int_{-b}^{b}\left|\frac{f(x)}{x}-\sum_{k=0}^{N}a_kP_{2k+1}(x)\right|^2w(x)\, dx <\epsilon,
\] which implies,
\begin{align*}
\int_{-b}^{b}\left|\frac{1}{x}\left(f(x)-\sum_{k=0}^{N}a_kxP_{2k+1}(x)\right)\right|^2w(x)\, dx 
&=\int_{-b}^{b}\left|f(x)-\sum_{k=0}^{N}a_kxP_{2k+1}(x)\right|^2\frac{w(x)}{x^2}\, dx\\
&=\int_{-b}^{b}\left|f(x)-\sum_{k=0}^{N}a_kQ_{2k+2}(x)\right|^2\frac{w(x)}{x^2}\, dx\\
&<\epsilon.
\end{align*}Therefore, $\{Q_{2k}(x)\}_{k=1}^{\infty}$ is complete in $\mathcal{S}$.
\qed

 \vspace{7mm}

\noindent {\bf Acknowledgments.} 
Research was supported by the Department of Mathematical Sciences at Bentley University. RB would like to thank Maxim Derevyagin for the helpful discussions and continued support.


\begin{thebibliography}{99}


\bibitem{Marcellan2010} M. Alfaro, F. Marcellan, A. Pena, and M. Luisa Rezola.\textit{When do linear combinations of orthogonal polynomials yield new sequences of orthogonal polynomials?}  Journal of Computational and Applied Mathematics 233.6 (2010), 1446–1452. 

\bibitem{Marcellan2011}M. Alfaro, A. Peña, M. L. Rezola and F. Marcellán. \textit{Orthogonal polynomials associated with an inverse quadratic spectral transform}, Comput. Math. Appl. 61 (2011), no. 4, 888–900.


\bibitem{Adler}V.E. Adler, On a modification of Crum’s method, Teor. Mat. Fiz. 101 (3) (1994) 323–330, (Russian); translation in Theoret. and Math. Phys. 101 (1994), no. 3, 1381–1386 (1995).

\bibitem{BD23}R. Bailey and M. Derevyagin. \textit{Complex Jacobi matrices generated by Darboux transformations.} J. Approx. Theory 288 (2023), Paper No. 105876, 33 pp.

\bibitem{Bailey2024}
R. Bailey and M. Derevyagin. \textit{DEK-type orthogonal polynomials and a modification of the Christoffel formula}. J. Comput. Appl. Math. 438 (2024), Paper No. 115561.

\bibitem{CG24}M. Castro, F.A. Gr\"unbaum. \textit{A new commutativity property of exceptional orthogonal polynomials}. RACSAM 2024, 118, Paper No. 81, 14 pp.

\bibitem{CGZ24}M. Castro, F. A. Grünbaum and I. Zurrián, \textit{Time and band limiting for exceptional polynomials}. Appl. Comput. Harmon. Anal. 2024, 68, 101600, 20 pp.

\bibitem{Chihara}T.S Chihara. {\it  An introduction to orthogonal polynomials.} Mathematics and its Applications, Vol. 13, Gordon and Breach Science Publishers, New York-London-Paris, 1978. xii+249 pp. 

\bibitem{ChiharaQuasi}T. S. Chihara. {\it On quasi-orthogonal polynomials.} Proc. Amer. Math. Soc. 8 (1957), 765–767.

\bibitem{MarcellanBueno}M.I.Bueno, F. Marcell\'an. \textit{Daroubx transformation and perturbation of linear functionals}. Linear Algebra Appl., 384 (2004), 215-242.

\bibitem{MJ26}M. Derevyagin, J.C. García-Ardila. \textit{A unified approach via Geronimus transformation to various types of orthogonal polynomials.} Preprint, arXiv:2608.28414 [math.CA], 2026

\bibitem{DGM14}M. Derevyagin, J.C. García-Ardila, F. Marcellán.
\textit{Multiple Geronimus transformations.}
Linear Algebra Appl., 454 (2014), 158-183.


\bibitem{DM13}M. Derevyagin, F. Marcell\'an. \textit{A note on the Geronimus transformation and Sobolev orthogonal polynomials}. Numer. Algorithms. 67 (2014), 271–287. 


\bibitem{DEK92} S.Yu. Dubov, V.M. Eleonskii, and N.E Kulagin, \textit{Equidistant Spectra of Anharmonic Oscillators}.  Soviet Phys. JETP 75 (1992), no. 3, 446--451; translated from Zh. Èksper. Teoret. Fiz. 102 (1992), no. 3, 814--825 (Russian).

\bibitem{DEK94} S.Yu. Dubov, V.M. Eleonskii, and N.E Kulagin, \textit{Equidistant spectra of anharmonic oscillators}. Chaos 4 (1994), no. 1, 47–53. 

\bibitem{Duran25}
A.J. Durán, \textit{Zeros of linear combinations of orthogonal polynomials}.(2025). arXiv preprint arXiv:2505.11956.

\bibitem{D21}
A. Dur\'an, \textit{Exceptional orthogonal polynomials}.
Lectures on orthogonal polynomials and special functions, 1-75, London Math. Soc. Lecture Note Ser., 464, Cambridge Univ. Press, Cambridge, 2021.

\bibitem{Fejer33}L. Fej\'er. \textit{Mechanische Quadraturen mit positiven Cotesschen Zahlen}, Math. Z. 37 (1933), 287-309.



\bibitem{GGM19}
M.A. Garc\'ia-Ferrero, D. G\'omez-Ullate, R. Milson, \textit{A Bochner type characterization theorem for exceptional orthogonal polynomials}. J. Math. Anal. Appl. 472 (2019), no. 1, 584--626.

\bibitem{DDM}D. G\'omez-Ullate, Y. Grandati, and R. Milson, \textit{Rational extensions of the quantum harmonic oscillator and exceptional Hermite polynomials}, J. Phys. 47 (2014), no. 1, 015203, 17 pp.


\bibitem{GM20}
D. G\'omez-Ullate, R. Milson, \textit{Exceptional orthogonal polynomials and rational solutions to Painlev\'e equations}. In Orthogonal polynomials, 335-386, Tutor. Sch. Workshops Math. Sci., Birkh\"auser/Springer, Cham, 2020.

\bibitem{Grinshpun2004} Grinshpun, Z. \textit{Special linear combinations of orthogonal polynomials}. J. Math. Anal. Appl., 299 (2004), 1-18.


\bibitem{GH97} F. A. Gr\"unbaum, L. Haine, \textit{Bispectral Darboux transformations: an extension of the Krall polynomials}. Int. Math. Res. Notices 1997, no. 8, 359–392.




\bibitem{I09}
M. E. H. Ismail, \textit{Classical and quantum orthogonal polynomials in one variable}. With two chapters by Walter Van Assche. With a foreword by Richard A. Askey. Reprint of the 2005 original. Encyclopedia of Mathematics and its Applications, 98. Cambridge University Press, Cambridge, 2009


\bibitem{maroni}P. Maroni. \textit{Sur la suite de polyn\'omes orthogonaux associ\'ee \`a la forme $u=\delta c+\lambda (x-c)-L$}, Periodica Mathematica Hungarica 21 (1990): 223-248.


\bibitem{Milson2009}D. Gómez-Ullate, N. Kamran, and R. Milson. \textit{An extended class of orthogonal polynomials defined by a Sturm–Liouville problem}, J. Math. Anal. Appl. 359 (2009), no. 1, 352–367. MR2511475.


\bibitem{shohat}J. Shohat. \textit{On mechanical quadratures, in particular, with positive coefficients}. Trans. Amer. Math. Soc. 42 (1937), 461-496.


\bibitem{Zhedanov97}A. Zhedanov. \textit{Rational spectral transformations and orthogonal polynomials}, J. Comput. Appl. Math. 85 (1997) 67–86.
\end{thebibliography}
\end{document}